\documentclass[14pt, dvipsnames]{extarticle}
\usepackage[left=2cm, top=2cm, right=2cm, bottom=20mm]{geometry} 

\usepackage[T1]{fontenc} % output T1 font encoding (8-bit) for accented characters as single glyph
\usepackage[cmintegrals,cmbraces]{newtxmath}
\usepackage[lining]{ebgaramond}
\usepackage{ebgaramond-maths}
\usepackage{euscript}
\usepackage{xcolor}
\usepackage{graphicx}
\usepackage[figurename=Fig.]{caption}

\usepackage{amsthm} % Mathematical typesetting
\usepackage{hyperref} % For hyperlinks in the PDFv

\theoremstyle{plain}
\newtheorem{theorem}{Theorem}

\newtheorem{lemma}{Lemma}
\newtheorem{prop}{Proposition}

\theoremstyle{definition}
\newtheorem{defi}{Definition}

\newtheorem{example}{\bf Example}

\theoremstyle{remark}

\usepackage{enumerate}

\let\phi\varphi

\DeclareMathOperator{\G}{\mathbb{G}}

\DeclareMathOperator{\F}{\EuScript{F}}
\DeclareMathOperator{\sym}{\mathrm{sym}}
\DeclareMathOperator{\nGrp}{\mathit{n}-\EuScript{G}\mathit{rp}}
\DeclareMathOperator{\inv}{\mathrm{inv}}
\DeclareMathOperator{\Aut}{\mathrm{Aut}}
\DeclareMathOperator{\im}{\mathrm{im}}
\DeclareMathOperator{\Z}{\mathbb{Z}}
\DeclareMathOperator{\Cc}{\mathbb{C}}
\DeclareMathOperator{\N}{\mathbb{N}}

 \hypersetup{
    colorlinks,
    linkcolor={red!50!black},
    citecolor={blue!50!black},
    urlcolor={blue!80!black}
}

\usepackage{stackrel}
\usepackage[all]{xy}

\usepackage{bm}

\usepackage{indentfirst}

\title{$n$-valued Coset Groups and Dynamics}
\author{Mikhail Kornev}
\date{}

\begin{document}

\maketitle

\begin{abstract}
We obtain asymptotic and exact formulae of growth functions for some families of $n$-valued coset groups. We also describe connections between the theory of $n$-valued groups and Symbolic Dynamics.
\end{abstract}

\tableofcontents

\section{Introduction}

In 1971, S. P. Novikov and V. M. Buchstaber gave the construction, which was predicted by the theory of characteristic classes. This construction describes a multiplication, with a product of any pair of elements being a non-ordered multiset of $n$ points. It led to the notion of $n$-valued groups, which was given axiomatically and developed by V. M. Buchstaber. Many authors are studying the theory of $n$-valued groups with applications in various areas of Mathematics and Mathematical Physics. Since 1996, V. M. Buchstaber and A. P. Veselov have been studying the applications of the theory of $n$-valued groups to discrete dynamical systems \cite{Buchstaber_Veselov1}. In 2010, V. Dragović showed that the associativity equation for some 2-valued group explains the Kovalevskaya top integrability mechanism \cite{Dragovich1}. Recently, M. Chirkov obtained results about the growth function of some $n$-valued group on complex numbers $\Cc$ \cite{Chirkov}. He also established connections between the compositions of integer numbers and Venn diagrams. See references for further relevant results.  

In Section \ref{preliminaries}, we briefly introduce the main constructions from the theory of $n$-valued groups and fix some notations following \cite{Buchstaber}. 

In Section \ref{growth}, we get exact and asymptotic formulae for the growth functions of the $n$-vaued coset groups $$\G_\phi(\langle a, b\ | \ a^m = b^m = 1 \rangle)\text{ and }\G_\phi\left ( \langle a_1, ..., a_s \ | \ a_1^2=...=a_s^2=1 \rangle \right ),$$ where $\phi$ is the automorphism permuting cyclically the generators of group presentations in each case. Surprisingly, this problem is related to $n$-bonacci numbers, and we use the result of \cite{Du} to succeed.  

In Section \ref{symbolic}, we consider more closely the case $\G_\phi(\langle a, b\ | \ a^3 = b^3 = 1 \rangle)$, constructing a fruitful tree which enables us to establish an interesting combinatorial observation in Proposition \ref{observation}, and connections with Symbolic Dynamics.

\section*{Acknowledgments}

The author is grateful to Dr. Victor Buchstaber for the inspiring atmosphere and fruitful discussions during the preparation of this paper.

\section{Preliminaries}\label{preliminaries}

Let us give some definitions and examples following paper \cite{Buchstaber}.

\begin{defi}
An \emph{$n$-valued multiplication} on a set $X$ is a map
\[
\nu: X\times X \to \sym^n(X)\,:\,\nu(x,y)=x*y=[z_1,z_2,\dots, z_n], \; z_k=(x*y)_k,
\]
where $\sym^n(X):=X^n/\Sigma_n$ is the $n$-symmetric power of $X$, or, equivalently, the set consisting of all $n$-multisets $[x_1, ..., x_n]$ (or, sets for short), such that the following conditions are satisfied:

\begin{itemize}

\item {\it Associativity.} The $n^2$-sets
\begin{gather*}[x*(y*z)_1,x*(y*z)_2,\dots,x*(y*z)_n],\\
[(x*y)_1*z,(x*y)_2*z,\dots,(x*y)_n*z]
\end{gather*}
are equal for all $x,y,z\in X$.

\item \emph{Unit.} There is an element $e\in X$, for which $e*x=x*e=[x,x,\dots,x]$\; for all\; $x\in X$.

\item \emph{Inverse map.} There is a map $\mathrm{inv}\colon X\to X$ with the property

\[
e\in \mathrm{inv}(x)* x\text{ and }\; e \in x*\mathrm{inv}(x)\text{ for all }x\in X.
\]
\end{itemize}
If a set $X$ has an $n$-valued multiplication, then $X$ is called an {\it $n$-valued group}. Sometimes, we denote $n$-valued groups via blackboard bold, like $\G$.
\end{defi}

\begin{example}
Any 1-valued group is an ordinary group. 
\end{example}

\begin{example}\label{nonnegative}

Consider the semigroup of nonnegative integers $\mathbb{Z}_+$. Define the multiplication $\nu\colon \mathbb{Z}_+\times \mathbb{Z}_+\to \sym^2\mathbb{Z}_+$
by the formula $x*y=[x+y,|x-y|]$. The unit $e=0$. The inverse map $\mathrm{inv}(x)=x$. For the associativity, one has to verify that the two $4$-sets
\[[x+y+z, |x-y-z|,x+|y-z|, |x-|y-z||]\]
and
\[[x+y+z,|x+y-z|,|x-y|+z,||x-y|-z|]\]
are equal for all nonnegative integers $x,y,z$ (an exercise for the reader).
\end{example}

As in the case of ordinary groups, one can define homomorphisms of $n$-valued groups. Hence, the class of all $n$-valued groups forms a category $\nGrp$.

\begin{defi}
A map $f: X\to Y$ of $n$-valued groups is called \emph{a homomorphism} if 

\begin{itemize}

\item $f(e_X)=e_Y$.

\item $f(\mathrm{inv}_X(x))=\mathrm{inv}_Y(f(x))$ for all $x\in X$.

\item $\nu_{Y}(f(x),f(y))=(f)^n\nu_X(x,y)$ for all $x, y\in X$.

\end{itemize}

In particular, the following diagram commutes:

\[
\xymatrix{
    X\times X \ar[r]^{\nu_X} \ar[d]_{f\times f} & \sym^n(X) \ar[d] \\
    Y\times Y \ar[r]       & \sym^n(Y) }
\]

\end{defi}

There is a way to construct many $n$-valued groups \cite{Buchstaber}. Given an ordinary ($1$-valued) group $G$ with the mutliplication $\nu_0$, the unit $e_G$, and the inverse element $\inv_G(u) = u^{-1}$, and given a finite subgroup $H$ of automorphism group $\Aut(G)$. Consider the quotient space $X := G/\phi(H)$ of $G$ by the action of the group $\im(\phi)$ with the quotient map $\pi: G\to X$ and $n$-valued multiplication $$\nu: X\times X\to \sym^n(X)$$ defined by the formula
\begin{equation}\label{coset}
\nu(x, y) = [ \pi(\nu_0(u, h(v))) \ | \ h\in H ]
\end{equation}
where $u\in\pi^{-1}(x)$ and $v\in\pi^{-1}(y)$.

\begin{theorem}[V. Buchstaber, \cite{Buchstaber}]\label{coset_theorem}
The multiplication $\nu$ from {\normalfont (\ref{coset})} defines an $n$-valued group structure on the orbit space $X = G/\phi(H)$ with the unit $e_x = \pi(e_G)$ and the inverse map $\inv_X(x) = \pi(\inv_G(u))$, where $u\in\pi^{-1}(x)$.
\end{theorem}

\begin{defi}
The $n$-valued group $\G_\phi(G)$ from Theorem \ref{coset_theorem} is called the {\it coset group}.
\end{defi}

\begin{example}\label{Z_2}
Consider $G = \Z/2\ast\Z/2 \cong \langle a, b\ | \ a^2 = b^2 = e \rangle$ and the finite cyclic group of automorphisms generated by the element $\phi: a\mapsto b$ of order $2$. It is easy to see that the underlying set of $X = G/\langle \phi \rangle$ consists of elements $$u_{2n}=[(ab)^n, (ba)^n],$$ $$u_{2n+1}=[b(ab)^n, a(ba)^n]$$ for every $n\geqslant 0$, and the 2-valued mutliplication in the coset group $\G_\phi(G)$ being $$u_k\ast u_\ell = [u_{k+\ell}, u_{|k-\ell|}].$$ Particularly, the 2-valued groups $\G_\phi(G)$ and $\Z_{+}$ from Example \ref{nonnegative} are isomorphic. It means that $\Z_{+}$ is a $2$-valued coset group (but not every $n$-valued group is a coset group).  
\end{example}

It is time to discuss the "bridge" between the theory of $n$-valued groups and dynamical systems. 

\begin{defi}
An {\it $n$-valued dynamics} $T$ on a space (not necessarily an $n$-valued group) $X$ is a map $T: X\to \sym^n(X)$.
\end{defi} 

Thinking about $X$ as a state space, an $n$-valued dynamics $T$ defines possible states $T(x) = [x_1, ..., x_n]$ at the moment $t+1$ as a state function of $x$ at the moment $t$.

\begin{example}
Consider a polynomial with respect to the variable $y$: $$F(x, y) = b_0(x)y^n + b_1(x)y^{n-1}+...+b_n(x), \ x,y\in \Cc,$$ with $b_i: \Cc\to \Cc$ being arbitrary functions. The equation $F(x, y) = 0$ defines an $n$-valued dynamics $$T: \Cc\to \sym^n(\Cc), \ x\mapsto [y_1, ..., y_n],$$ where $[y_1, ..., y_n]$ is the $n$-set of roots of $F(x, y) = 0$.
\end{example}

For an $n$-valued dynamics $T: X\to \sym^n(X)$, there is a fruitful characteristic measuring the growth of $T$.

\begin{defi}\label{dynamics}
For any $x\in X$ and an $n$-valued dynamics $T: X\to \sym^n(X)$, the {\it growth function} $\xi_{x}:\N\to \N$ sends $k$ to the number of different points in $$\bigcup\limits_{i = 0}^k T^{i}(x) = \{\ast\}\cup T(x)\cup T^2(x)\cup ...\cup T^k(x),$$ where $\ast$ denotes a marked point in the underlying set of $X$. 
\end{defi} 

In this context, there is a general problem of characterizing such $n$-valued dynamics that functions $\xi_x$ have a polynomial growth (see Figure \ref{poly_exp}).

\begin{figure}
\centering
\includegraphics[width=0.6\linewidth]{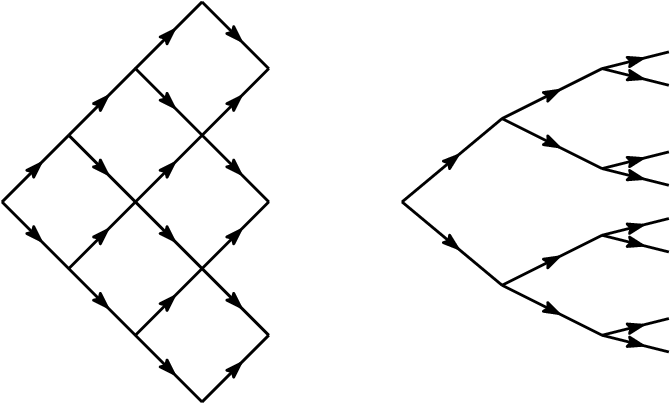}
\caption{Polynomial and exponential growth}
\label{poly_exp}
\end{figure} 

$n$-valued dynamical systems come, for example, from an action of $n$-valued groups on a space:

\begin{defi}\label{group_action}
An {\it action of $n$-valued group $\G$ on a space} (not necessarily an $n$-valued group) $X$ is defined to be the map $$f: \G\times X\to \sym^n(X),\ f(g, x) = g\cdot x = [x_1, ..., x_n],$$ such that

\begin{itemize}

 \item for any elements $g_1$, $g_2\in \G$ and any $x\in X$ the following $n^2$-multisets coincide: $$g_1(g_2x) = [g_1x_1, ..., g_1x_n]$$ and $$(g_1g_2)x = [h_1x, ..., h_nx],$$ where $g_2x = [x_1, ..., x_n]$ and $g_1g_2 = [h_1, ..., h_n]$,
 
 \item $ex = [x, ..., x]$ for any $x\in X$. 
 
 \end{itemize}
 
\end{defi}  

\begin{example}[The left action of $\G$ on itself]
In the setting of Definition \ref{group_action}, take $X$ and $f$ to be the group $\G$ and the operation on $\G$, respectively.
\end{example}

The next class of $n$-valued groups will be in the focus of this paper.

\begin{defi}\label{cyclic}
An $n$-valued group $\G$ is called {\it cyclic} if there is an element $g\in \G$, such that any element $h$ of the underlying set of $\G$ belongs to a multiset $g^k$ for some $k>0$.
\end{defi}

\section{Growth of $n$-valued groups}\label{growth}

In this section, we consider the class of such finitely presented groups $G = \langle f_1, ..., f_n\ | \ R \rangle$ that admit an automorphism $\phi\in\Aut(G)$ cyclically permuting $n$ generators $f_1, ..., f_n$. Let $\mathbb{G}_\varphi(G)$ denote the $n$-valued group arising from the group $G$ and the automorphism $\phi\in\Aut(G)$. Note that $\G_\phi(G)$ is cyclic in the sense of Definition \ref{cyclic}. We study the growth function $\xi_g(k)$ associated with the element $g = [f_1, ..., f_n]$. In this setting, we assume the marked point from Definition \ref{dynamics} to be the empty word $\Lambda$, or equivalently, the identity element of $G$. For short, we will use the notation $\xi_k$ instead of $\xi_g(k)$.

\begin{example}
In Example \ref{Z_2}, we have the linear growth $\xi_k = k + 1$.
\end{example} 

\begin{prop}\label{z3z3}
For the group $\mathbb{Z}/3\ast\mathbb{Z}/3 = \langle a, b \ | \ a^3=b^3=1 \rangle$ and the automorphism $a\mapsto b$ the corresponding $2$-valued coset group has the growth function $$\xi_k = F_{k+3} = \cfrac{1}{\sqrt{5}}\left ( \left(\cfrac{1 + \sqrt{5 }}{2} \right)^{k+3} - \left( \cfrac{1-\sqrt{5}}{2} \right )^{k+3} \right ) - 1.$$ In particular, the growth is exponential: $\xi_k \sim \cfrac{\varphi^{k+3}}{\sqrt{5}}$ for $k\to\infty$, where $\varphi = \cfrac{1+\sqrt{5}}{2}$.
\end{prop}

\begin{proof}

We will count the number of words of length $k$ at the step at which $[a, b]^k$ appears. At step zero, we have an empty word corresponding to the unit of the 2-valued group. The first step produces the element $[a, b]$. At the second step, we have: $[a, b]^2 = [[a^2, b^2], [ab, ba]]$ --- two elements.

At step $k$, all words in the normal form of the group $\mathbb{Z}/3\ast\mathbb{Z}/3$ will appear. Any such word has the form (up to automorphism) $$a^{k_1}b^{k_2}a^{k_3}...,$$ where $k_i\in \{1, 2\}$. It means there will be as many such words of length $k$ as there are in how many ways the number $k$ can be divided into summands from the set $\{1, 2\}$, taking into account the order. Let us denote the last quantity by $S_k$.

From step $k$, you can go to step $k-1$, taking $1$ as the first summand among $k_i$'s --- in this case, the remaining summands can be chosen in $S_{k-1}$ ways, or you can take $2$ as the first summand --- in this case, the ones can be chosen in $S_{k-2}$ ways. It means we have the recurrence relation $$S_k = S_{k-1}+S_{k-2}, S_0 = 1, S_1 = 1.$$

Recall that the recurrent sequence $$F_n = F_{n-1}+F_{n-2}$$ with initial conditions $$F_0 = 0, F_1 = 1,$$ is called the {\it Fibonacci sequence}. There is the well-known Binet formula $$F_n = \frac{1}{\sqrt{5}}\left( \left(\frac{1+\sqrt{5}}{2}\right)^n - \left(\frac{1-\sqrt{5}}{2}\right)^n\right ).$$

In our case, $$S_k = F_{k+1}.$$

It is easy to prove by induction that $$F_1 + ... + F_k = F_{k+2} - 1.$$

Hence $$\xi_k = S_0 + S_1 + ... + S_k = F_{k+3} - 1 = \cfrac{1}{\sqrt{5}}\left ( \left(\cfrac{1 + \sqrt {5}}{2} \right)^{k+3} - \left( \cfrac{1-\sqrt{5}}{2} \right )^{k+3} \right ) - 1.$$

\end{proof}

\begin{prop}\label{triple_z_two}
For the group $\mathbb{Z}/2\ast\mathbb{Z}/2\ast\mathbb{Z}/2 = \langle a, b, c \ | \ a^2=b^2=c^2=1 \rangle$ with the cyclic automorphism $a\mapsto b\mapsto c$, the corresponding $3$-valued group has the growth function $\xi_k = 2^k$. In particular, the growth is exponential.
\end{prop}

\begin{proof}

The normal form, up to automorphism, for the group from the conditions has the form in which the first letter of the word is the letter $a$, and identical letters are not adjacent. Let us prove by induction that the number $S_k$ of distinct words in the normal form of length $k$ in the multiset $[a, b, c]^k$ is equal to $2^{k - 1}$ for $k>0$, and $S_0 = 1$. When $k=0$, we have the empty word. Suppose that at step $k-1$, we have $2^{k-2}$ different words of length $k-1$. Then, we add one letter to the right of each word of length $k-1$ so that the length of the normal form of the word increases by $1$. It can be done in two ways for one fixed word. There are exactly $2^{k-2}$ words of length $k-1$, so there will be $2\cdot 2^{k-2} = 2^{k - 1}$ words of length $k$. The statement about $\xi_k$ now follows easily.

\end{proof}

\begin{prop}
For the group $\left ( \mathbb{Z}/2\right )^{\ast s}= \langle a_1, ..., a_s \ | \ a_1^2=...=a_s^2=1 \rangle$ with the cyclic automorphism $\phi: a_i\mapsto a_{i+1}$ $($the indices are taken modulo $s)$ the corresponding s-valued coset group has the following growth function $$\xi_{k}=\begin{cases}
\dfrac{\left(s-1\right)^{k}-1}{s-2}+1, & s\geqslant3\\
k + 1, & s=2
\end{cases}$$ In particular, the growth is polynomial if and only if $s = 2$.
\end{prop}

\begin{proof}
Similar to the proof of Proposition \ref{triple_z_two}.
\end{proof}

Before moving on to the next example, let us define a sequence of $n$-bonacci numbers.

\begin{defi}
A sequence of integers $\{F_k^{(n)}\}$ is called {\it $n$-bonacci} if it satisfies the recurrence relation 
\begin{equation}\label{n-bonacci}
F_k^{(n)} = F_{k-1}^{(n)} + ... + F_{k-n}^{(n)}
\end{equation} 
and initial conditions $F_0 = ... = F_{n-2} = 0, \ F_{n-1} = 1.$
\end{defi}

In what follows, instead of the cumbersome notation $F_n^{(n)}$, we will write $F_n$.

Using the standard recurrent sequence technique, we can obtain the following

\begin{prop}[see, for example, in \cite{Du}]\label{Binet}
The explicit formula for the $k$-th term of the $n$-bonacci sequence $($generalized Binet's formula$)$ is as follows
\begin{equation}\label{sum} F_k = \sum\limits_{i = 1}^n \frac{\lambda_i-1}{(n+1)\lambda_i-2n}\lambda_i^{k-n+ 1},\end{equation} where $\{\lambda_i\ | \ 1\leqslant i\leqslant n\}$ is the set of roots of the characteristic polynomial $\chi(\lambda) = \lambda^n - \lambda^{n-1} - ... - 1$.
\end{prop}

\begin{proof}
From the theory of linear recurrences with constant coefficients, it is known that the general solution of equation (\ref{n-bonacci}) has the form $$F_k = C_1\lambda_1^k + ... +C_n\lambda_n^k,$$ where $\lambda_i$ are the same as in the conditions. It remains to find the constants $C_i$ from the linear system of equations $$\begin{cases}
C_{1}+...+C_{n} & =0,\\
C_{1}\lambda_{1}+...+C_{n}\lambda_{n} & =0,\\
...\\
C_{1}\lambda_{1}^{k-2}+...+C_{n}\lambda_{n}^{k-2} & =0,\\
C_{1}\lambda_{1}^{k-1}+...+C_{n}\lambda_{n}^{k-1} & =1
\end{cases}$$

It can be easily done using Cramer's rule and the properties of Vandermonde's determinant. The system has a unique solution since the characteristic polynomial $\chi(\lambda)$ has no multiple roots. Due to the symmetry of the system for cyclic permutations, we restrict ourselves to finding $C_1$. We have $$C_1 = \frac{W_1}{W} =(-1)^{n+1} \frac{\prod\limits_{1<i<j\leqslant n}(\lambda_j - \lambda_i)}{ \prod\limits_{1\leqslant i < j \leqslant n}(\lambda_j - \lambda_i)} =\frac{(-1)^{n+1}}{\prod\limits_{1< j\leqslant n } (\lambda_j - \lambda_1)} = \frac{1}{\chi' (\lambda_1)},$$ where $W$ is the Vandermonde determinant resulting from the system above, $W_1$ is the determinant of the system with the first column replaced by a column of free terms of the system (as usual, in Cramer's rule).

It remains to find the value of $\chi'(\lambda_1)$. To do this, consider the polynomial $$P(\lambda) = (\lambda - 1)\chi(\lambda) = \lambda^{n+1} -2\lambda^n + 1.$$ It is clear that $$P '(\lambda_1) = (\lambda_1 - 1)\chi'(\lambda_1).$$ On the other hand, $$P'(\lambda_1) = (n+1)\lambda_1^n - 2n\lambda_1^{ n-1}.$$ Hence, taking into account the fact that $$\lambda_1^{n-1} = \frac{1}{\lambda_1(2 - \lambda_1)} = \frac{1}{\lambda_1\cdot \lambda_1^{-n}} = \frac{1}{\lambda_1^{-n+1}}, \text{ (since $2-\lambda_1 = \lambda_1^{-n})$} $$ it follows that $$C_1 = \frac{1}{\chi'(\lambda_1)} = \frac{\lambda_1-1}{(n+1)\lambda_1-2n}\lambda_1^{-n+1 }.$$
\end{proof}

From Rouché's theorem and Descartes' rule of signs, it easily follows that the polynomial $$\chi(\lambda) = \frac{\lambda^{n+1}-2\lambda^n+1}{\lambda-1}$$ from Proposition \ref{Binet} ($\lambda\neq 1$) has exactly one positive root greater than $1$ and $n$ complex roots inside the unit circle centered at zero on the complex plane. It means that among the terms of sum (\ref{sum}), there is precisely one whose absolute value is greater than $1$. It turns out that the sum of all other terms is always less than $1/2$ (see \cite{mathe} for the connection of this fact with random walks on the number line). Hence, the following holds:

\begin{theorem}[Z. Du, G. P. Dresden'14, \cite{Du}]\label{formula}
For the $k$th term of the $n$-bonacci sequence we have $$F_k^{(n)} = \mathrm{rnd}\, \left ( \frac{r-1}{(n+1)r-2n }r^{k-n+1} \right ),$$ where $r$ is the positive root of the polynomial $\chi(\lambda) = \lambda^n - \lambda^{n-1} - .. . - 1$, and $\mathrm{rnd}\, (\cdot)$ is the nearest integer.
\end{theorem}

Apply this result to the case of interest to us.

\begin{prop}
For the group $\mathbb{Z}/m\ast\mathbb{Z}/m = \langle a, b\ | \ a^m = b^m = 1 \rangle$ with the cyclic automorphism $\varphi: a\mapsto b$ of order $2$, the growth function $$\xi_{k}\sim\frac{r^{k+ 1}}{mr-2(m-1)},$$ for $m\geqslant 3$, $k\to\infty$, and $r$ is the positive root of the polynomial $\chi(\lambda) = \lambda^n - \lambda^ {n-1} - ... - 1$. In particular, the growth of the corresponding $2$-valued coset group is polynomial if and only if $m = 2$.
\end{prop}

\begin{proof}
The normal form in this case looks like $$a^{k_1}b^{k_2}a^{k_3}...,$$ where $k_i\in \{1, 2, ..., m-1\}$. Suppose that $m\geqslant 3$. Let $S_k$ denote the number of distinct words of length $k$ in the multiset $[a, b]^k$. Assume that $$S_{-(m-2)}=...=S_{-1} = 0 \text{ and } S_0 = 1.$$ We have the recurrence relation $$S_k = S_{k - 1 } + S_{k - 2} + ... + S_{k - m + 1}.$$ This recurrence relation, together with the initial conditions, determines the $(m-1)$-bonacci sequence shifted by $m-2$ $$S_k^{(m-1)} = F_{k+m-2}^{(m-1)},$$ where $k\geqslant -(m-2).$

Thus, by Theorem \ref{formula} $$S_k = \mathrm{rnd}\, \left ( \frac{r-1}{mr-2(m-1)}r^{k} \right ). $$

Since $r > 1$ (see the discussion after the proof of Proposition \ref{Binet}), we get: $$\xi_k = S_1 + S_2 + ... + S_k \sim \sum\limits_{i = 1}^k \frac{r-1}{mr-2( m-1)}r^{i} \sim \frac{r^{k+1}}{mr-2(m-1)}$$ for $k\to\infty$.

\end{proof}

\section{Connection with Symbolic Dynamics}\label{symbolic}

In this section, we more closely study the 2-valued group $\mathbb{G}_\varphi\left(\mathbb{Z}/3\ast\mathbb{ Z}/3\right)$ considered above in the context of combinatorics on words (see, e.g., \cite{Lotaire, Semenov}).

For short, let us give the following 

\begin{defi}
We will call {\it cubeless} any word in the normal form of the group $\mathbb{Z}/3\ast\mathbb{Z}/3 = \langle a, b\ | \ a^3=b^3=1 \rangle$.
\end{defi}

From the normal form of the group $\mathbb{Z}/3\ast\mathbb{Z}/3$, it is clear that at step $k$ all possible cubeless words in the alphabet $\{a, b\}$ appear.

Let us construct an oriented tree $\Gamma$ (see Figure \ref{tree2}), the vertices of which will contain elements of the 2-valued group $\mathbb{G}$ as follows. At step $0$, we start with the vertex corresponding to the empty word $\Lambda$ --- this is the root of our tree. At step $1$, we form a vertex $[a, b]$ adjacent to the root. At step $2$, we draw two edges from the last vertex, each corresponding to assigning the letter $a$ or $b$ to the right. You will get two vertices with words of length $2$: $[a^2, b^2]$ and $[ab, ba]$. At step $k$, we start with all cubeless words of length $k-1$ and draw $1$ or $2$ edges from each vertex, guided by the following principle: if a word ends with the first degree of a letter, then exactly $2$ edges will depart from this vertex, corresponding multiplication by $a$ or $b$; if a word ends with the square of a letter, then precisely one edge corresponding to this additional letter will depart from this vertex.

\begin{figure}[!h]
\begin{center}
\includegraphics[scale=0.6]{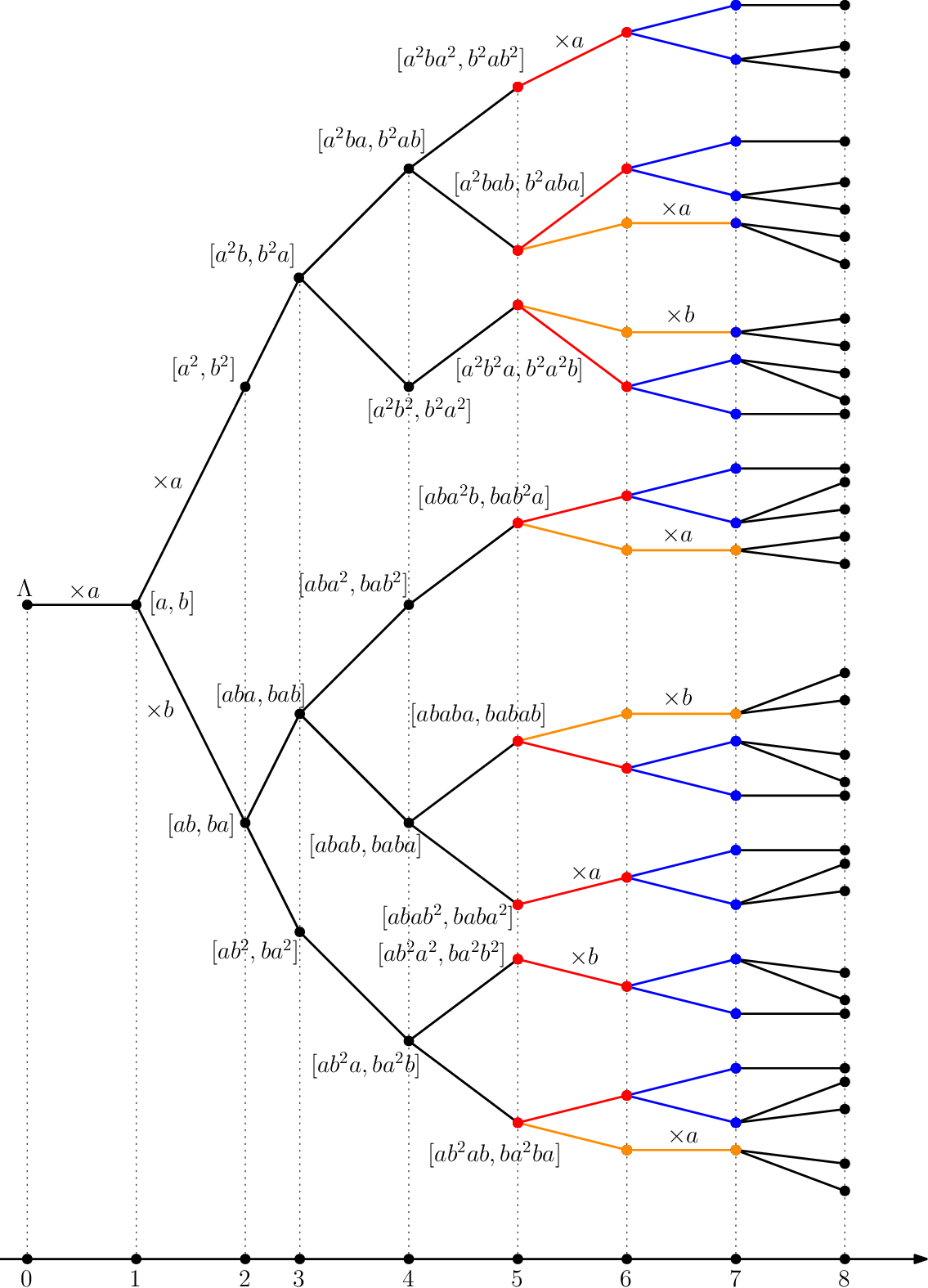}
\caption{}
\label{tree2}
\end{center}
\end{figure}

In this tree, at level $k$, there are $F_{k+1}$ vertices (as we saw in Proposition \ref{z3z3}).

Interestingly, one can draw two paths on this graph, corresponding to the two well-known infinite words in Symbolic Dynamics. To define them, introduce the notion of {\it morphism}. 

\begin{defi}
 Let $A$ and $B$ be alphabets. A {\it morphism} is a map $\F$ between the sets of words $A^\ast$ and $B^\ast$ in the corresponding alphabets satisfying $$\F(xy) = \F(x)\F(y)$$ for all words $x, y\in A^\ast$, i.e., $\F$ is a homomorphism of monoids.
\end{defi}  

In some cases, the sequence of words $\{\F^{n}(a)\}$ has a limit $\F^{\infty}(a)$ (see \cite{Lotaire} for details).

\begin{example}[\cite{Lotaire}]
Consider a morphism $$\F: \{a, b\}^\ast\to\{a, b\}^\ast,\ a\mapsto ab, \ b\mapsto a.$$ The {\it infinite Fibonacci word} is defined to be $\Phi := \F^{\infty}(a)$
$$\Phi = abaababaabaababaababaabaababaaba...$$
\end{example}

\begin{example}[\cite{Lotaire}]
A morphsim $$\F: \{a, b\}^\ast\to \{a, b\}^\ast, \ a\mapsto ab,\ b\mapsto ba$$ defines the \emph{Thue-Morse sequence} $\EuScript{T}:=\F^\infty(a)$

$$\EuScript{T} = abbabaabbaababbabaababbaabbabaab...$$
\end{example}

Since the paths in the tree $\Gamma$ correspond to all possible cubeless words, the Fibonacci word $\Phi$ and the Thue-Morse sequence $\EuScript{T} $ define some two paths. In Figure \ref{tree3}, the path corresponding to the Fibonacci word is shown in purple, and the path corresponding to the Thue-Morse sequence is shown in red (with the overlap along the first two edges starting from the root $\Lambda$).

\begin{figure}[h!]
\begin{center}
\includegraphics[scale=0.6]{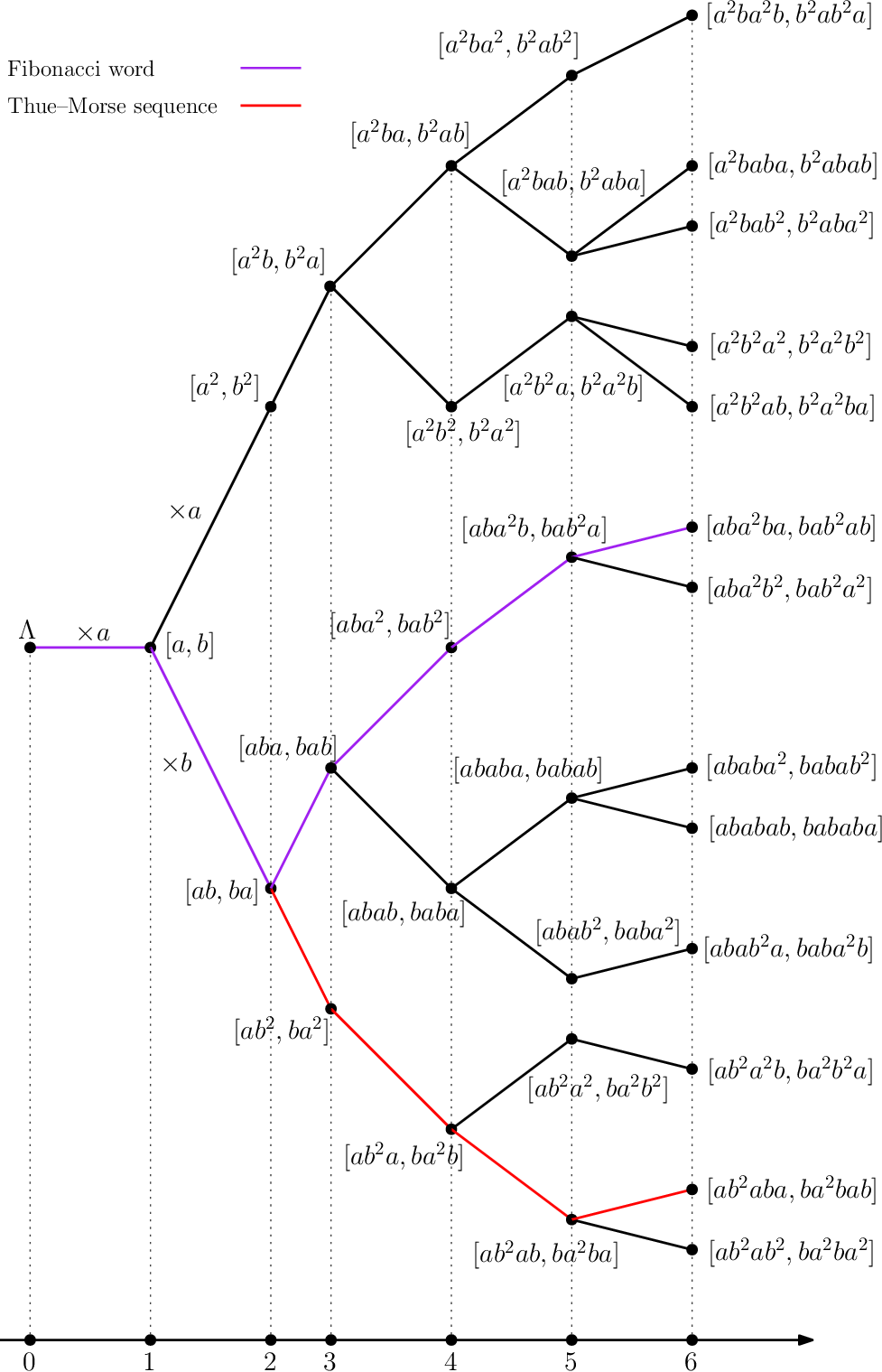}
\caption{}
\label{tree3}
\end{center}
\end{figure}

Note the following property of the constructed tree $\Gamma$.

\begin{lemma}
At the $k$-th level of the tree $\Gamma$, all cubeless words of length $k$ starting with the letter $a$ are located from top to bottom in the standard lexicographic order.
\end{lemma}

\begin{proof}
Induction on $k$. When $k = 0$, we have an empty word, and the statement is obvious. Consider the set $L_k$ of all vertices of level $k$. Let $v_1, v_2\in L_k$, and $v_1 < v_2$ in the lexicographic order. Then the inequality $$\mathrm{Child}(v_1) < \mathrm{Child}(v_2)$$ will be satisfied between the children of the vertices $v_1$ and $v_2$, respectively, i.e. $u_1 < u_2$ for any $u_1 \in\mathrm{Child}(v_1)$ and $u_2\in\mathrm{Child}(v_2)$. It is so because additions occur on the right, and words are compared on the left in the lexicographic order.
\end{proof}

\begin{lemma}\label{lemma4}
For an arbitrary vertex $v$ of the tree $\Gamma$, consider the subtree $\Gamma_v$ starting from this vertex. Then the numbers of vertices at each level of the subtree $\Gamma_v$ form the Fibonacci sequence starting either from the Fibonacci number $F_1$ or $F_2$ depending on whether the word $w(v)$ $($with the first letter $a)$ at the vertex $v$ ends with the first or second degree of a letter.
\end{lemma}

\begin{proof}
Let us look at the last letter of the word $w(v)$ and forget about the remaining letters. It is either the first power of a letter, say $a$, or the second power, say $a^2$ (without loss of generality). In the first case, a subtree $\Gamma_{[a, b]}$ will grow from $v$, and in the second case, the original tree $\Gamma_\Lambda = \Gamma$. But for the original tree $\Gamma$, the numbers of vertices at each level form the Fibonacci sequence by constructing $\Gamma$, as noted in Proposition \ref{z3z3}.
\end{proof}

From this lemma, we get

\begin{prop}\label{observation}
Consider the sequence $$\{\Theta_k\} = \{\Psi, \Psi a, \Psi aa, \Psi aab, ...\}$$ of subwords of the infinite cubeless word $$\Psi aabaabaab... =: \Psi(aab ),$$ where $\Psi$ is any cubeless word whose the last letter is different from the letter $a$. Then the number $Q_k$ of cubeless words which are not smaller than the word $\Theta_k$ in the standard lexicographic order satisfies the recurrence relation $$Q_k = Q_{k-1} + Q_{k-2}.$$
\end{prop}

\begin{proof}
The word $\Psi aabaabaab... = \Psi(aab)$ corresponds to an infinite path $\gamma$ in the tree $\Gamma$. Each $\Theta_k$ bijectively corresponds to some vertex on this path. The set of words $\bigcup\limits_{k} Q_k$ we are interested in corresponds to all vertices lying under $\gamma$. But the last set is the union of all the vertices of the subtrees. Now, the statement follows from the summation of all recurrence relations for each subtree according to Lemma \ref{lemma4}.
\end{proof}

\end{document}